\documentclass[a4paper,10pt]{amsart}
\usepackage{amssymb,amsthm,amsmath}
\usepackage{ifthen}
\usepackage{graphicx}
\usepackage{wasysym}
\usepackage{color}
\nonstopmode
 \numberwithin{equation}{section}
\newtheorem{thm}{Theorem}[section]
\newtheorem{cor}[thm]{Corollary}
\newtheorem{lem}[thm]{Lemma}
\newtheorem{prop}[thm]{Proposition}
\newtheorem{defn}[thm]{Definition}

\theoremstyle{definition}

\newtheorem{question}{Question}[section]

\newcommand{\norm}[1]{{\|#1\|}}
\newcommand{\eps}{{\varepsilon}}

\newenvironment{rem}{%
\bigskip
\noindent \textsl{{\sl Remark. }}}{\bigskip}

\newcommand{\ip}[2]{\langle#1,#2\rangle}

{\qed\bigskip}

\newcounter{alphabet}
\newcounter{tmp}

\newcommand{\vertiii}[1]{{\left\vert\kern-0.25ex\left\vert\kern-0.25ex\left\vert #1 
    \right\vert\kern-0.25ex\right\vert\kern-0.25ex\right\vert}}
  
\ifx\undefined\bysame
\newcommand{\bysame}{\leavevmode\hbox to3em{\hrulefill}\,}
\fi

\begin{document}
\baselineskip=21pt
\markboth{} {}

\bibliographystyle{amsplain}

\title[Optimal $\ell^1$ rank one matrix decomposition]
{Optimal $\ell^1$ rank one matrix decomposition}

\author{Radu Balan}
\author{Kasso A. Okoudjou}
\author{Michael Rawson}
\author{Yang Wang}
\author{Rui Zhang}

\address{Department of Mathematics, University of Maryland, College Park, MD 20742, USA}
\email{rvbalan@umd.edu}
\address{Department of Mathematics, University of Maryland, College Park, MD 20742, USA}
\email{okoudjou@umd.edu}
\address{Department of Mathematics, University of Maryland, College Park, MD 20742, USA}
\email{rawson@umd.edu}
\address{Hong Kong University of Science and Technology}
\email{yangwang@ust.hk}
\address{Hong Kong University of Science and Technology}
\email{zhangrui112358@yeah.net}

\date{\today}
\maketitle
\def\BC{{\mathbb C}} \def\BQ{{\mathbb Q}}
\def\BR{{\mathbb R}} \def\BI{{\mathbb I}}
\def\BZ{{\mathbb Z}} \def\BD{{\mathbb D}}
\def\BP{{\mathbb P}} \def\BB{{\mathbb B}}
\def\BS{{\mathbb S}} \def\BH{{\mathbb H}}
\def\BE{{\mathbb E}}
\def\BN{{\mathbb N}}
\def\LP{{W(L^p, \ell^q)(\mathbb{R}^d)}}
\def\L1{{W(\mathbb{R}^d)}}
\def\R{{\mathbb{R}}}
\vspace{-.5cm}

\begin{abstract}
In this paper we consider the decomposition of positive semidefinite matrices as a sum of rank one matrices. We introduce and investigate the properties of various measures of optimality of such decompositions. For some classes of positive semidefinite matrices we give explicitly these optimal decompositions. These classes include diagonally dominant matrices and certain of their generalizations, $2\times 2$, and a class of $3\times 3$ matrices. 
\end{abstract}

\section{Introduction}\label{sec1}

The finite dimensional matrix factorization problem that we shall investigate was partially motivated by a related infinite dimensional problem, which we briefly  recall.

Suppose that  $\BH$  is  an infinite-dimensional separable Hilbert space, with norm $\Vert \cdot \Vert$ and inner product $\langle \cdot , \cdot \rangle$.
Let $\mathcal{I}_1 \subset \mathcal{B}(\BH)$  be the subspace of trace-class operators.  For a detailed study on trace-class operators see \cite{dun88, sim79}.
Consider an orthonormal basis  $\{w_n\}_{n \geq 1}$ for $\BH$, and let 
$$ \BH^1=\Big\{ f \in \BH : \vertiii{f}:=\sum_{n=1}^{\infty} |\langle f, w_n\rangle| < \infty \Big\}.$$
For a sequence $c=(c_{mn})_{m, n=1}^\infty \in \ell^1$ we consider the operator   $T_c :\BH \rightarrow \BH$ given by  
$$T_cf =\sum_{m=1}^\infty \sum_{n=1}^\infty c_{mn} \ip{f}{w_n}w_m .$$ 
We say that $T_c$ is of \emph{Type $A$} with respect to the orthonormal basis $\{w_n\}_{n \geq 1}$ if, for an orthogonal set of eigenvectors $\{g_n\}_{n \geq 1}$ of $T_c$ such that $T_c=\sum_{n =1}^\infty g_n \otimes \overline{g_n}$, with convergence in the strong operator topology, we have that 
\begin{equation*}
\sum_{n=1}^{\infty}  \vertiii{g_n}^2 < \infty.
\end{equation*} 

Similarly, we say that the  operator $T_c$ is of \emph{Type $B$} with respect to the orthonormal basis $\{w_n\}_{n \geq 1}$  if there is some sequence of vectors $\{v_n \}_{n \geq 1}$ in $\BH$ such that $T_c=\sum_{n=1}^\infty v_n \otimes 
\overline{v_n}$ with convergence in the strong operator topology and we have that
\[ \sum_{n=1}^{\infty}  \vertiii{v_n}^2 < \infty. \] 
It is easy to see that if $T_c$ is of Type $A$ then it is of Type $B$. However, there exist finite rank positive trace class operators which are neither of Type $A$ nor of Type $B$. We refer to \cite{hei08} for more details. 
In \cite{BOP} we proved that there exist positive trace class operators $T_c$ of Type $B$ which are not of Type $A$. Furthermore, this answers negatively a problem posed by Feichtinger \cite{FeiOber}.

Our main interest is in a finite dimensional version of the above problem. Before stating it, we set the  notations that will be used through this chapter.

For $n\geq 2$ we denote  the set of all complex hermitian $n \times n$ matrices as $S^n:=S^n(\BC)$, positive semidefinite matrices as $S_+^n:=S^n_+(\BC)$, and positive definite matrices $S_{++}^n:=S^n_{++}(\BC)$.  It is clear that $S^n_+$ is a closed convex cone. Note that  $S^n=S^n_+-S^n_+$ is the (real) vector space of hermitian matrices. We will also use the notation $U(n)$ for the set of  $n\times n$ unitary matrices. 

For $A\in S^n$, we let $\|A\|_{1,1}=\sum_{k, \ell=1}^n|A_{k,  \ell}|$, and we let $\|A\|_{\mathcal{I}_1}=\sum_{k=1}^n|\lambda_k|$ where $ \lambda_1\leq \lambda_2\leq \hdots \leq \lambda_n$ are the eigenvalues of $A$. We recall that the operator norm of  $A\in S^n$ is given by $\|A\|_{\text{op}}=\max\{|\lambda_k|:, \lambda_1\leq \lambda_2\leq \hdots \leq \lambda_n\}$ where $\{\lambda_k\}_{k=1}^n$ is the set of eigenvalues of $A$. In addition, the Frobenius norm of $A$ is given by $\|A\|_{\text{Fr}}=\sqrt{\text{tr}AA^{*}}=\sqrt{\sum_{k=1}^n\sum_{\ell=1}^n|A_{k \ell}|^2}.$ One important fact that will be used implicitly throughout the paper is that all the norms defined on $S^n$ are equivalent and thus give rise to the same topological structure on $S^n$. 

Similarly, for a vector $x=(x_k)_{k=1}^n \in \BC^n$, and $p\in (0, \infty)$ we let $\|x\|_p^p=\sum_{k=1}^n|x_k|^p$ define the usual $\ell^p$ norm, $p\ge1$, with the usual modification when $p=\infty$, and $p=0$.  As pointed out above all these norms are equivalent on $\BC^n$ and give rise to the same topology.

The goal of this chapter is to investigate optimal decompositions of a matrix $A\in S^n_+(\BC)$ as a sum of rank one matrices.  In  Section~\ref{sec:2} we  introduce some measures of optimality of the kinds of  decompositions we seek, and investigate the relationship between these measures. However, before doing so,  we give an exact statement of the problems we shall address and  review some results about the convex cone $S^n_+(\BC)$.  In Section~\ref{sec:3} we restrict our attention to some classes of matrices in $S^n_+(\BC)$, including diagonally dominant matrices. Finally, in Section~\ref{sec:4} we report on some numerical experiments designed to find some of these optimal decompositions.

\section{Preliminaries and measures of optimality}\label{sec:2}
In the first part of this section, we collect some foundational facts on convex subsets of $S^n$. The second part will be devoted to introducing some quantities that will serve as measures of optimality of the decomposition results we seek.

\subsection{Preliminaries}\label{subsec:2.1}
We denote the convex hull of a set $S$ by $\text{co} S$. For the compact set  $X=\{xx^*: x \in \BC^n \mathrm{\; and \;} \Vert x \Vert_1=1\}$, we let $\Gamma=\text{co} X$ and $\Omega=\text{co}  \left(X \cup \{0\} \right)$. Observe that $\Omega \subset S^n_+(\BC)$. In fact, the following result holds. 

\begin{defn}
An extreme point is a point such that it is not a convex combination of other points.
\end{defn}

\begin{lem}\label{lem:omega}
$\Omega$ is closed and compact convex subset of $S^n_+(\BC)$ with int $\Omega \neq \emptyset.$  Furthermore, the set of extreme points of $\Omega$ is $X\cup\{0\}$. 
\end{lem}

The proof is based on one of the versions of the  Minkowski-Carath\'eodory Theorem, which, for completeness we recall. We refer to \cite{DLGMM19, FC13,rey65} for more details and background. 

\begin{thm}\cite[Proposition 3.1]{DLGMM19}\cite[Lemma 4.1]{rey65} \label{thm:MCT} 
(Minkowski-Carath\'eodory Theorem) Let A be a compact convex subset of a normed vector space $X$ of finite dimension $n$. Then any point in $A$ is a convex combination of at most $n+1$ extreme points. Furthermore, we can fix one of these extreme points resulting in expressing any point in $A$ is a convex combination of at most $n$ extreme points in addition to the one we fixed. 
\end{thm}

\begin{proof}[Proof of Lemma~\ref{lem:omega}]
$\Omega$ can be written as:
\begin{eqnarray*}
\Omega &=& \Big\{ \sum_{k=1}^m w_k x_k x_k^*: m\geq 1,\;\mathrm{an \; integer}, w_1,..,w_m \geq 0, \sum_{k=1}^mw_k \leq 1, \Vert x_k \Vert_1=1, 1\leq k \leq m  \Big\}\\
&=&\bigcup_{m \geq 1} \Big\{ \sum_{k=1}^m w_k x_k x_k^*: w_1,..,w_m \geq 0, \sum_{k=1}^mw_k \leq 1, \Vert x_k \Vert_1=1, 1\leq k \leq m  \Big\}\\
&=& \bigcup_{m \geq 1} \Omega_m,
\end{eqnarray*}
where $\Omega_m=\Big\{ \sum\limits_{k=1}^m w_k x_k x_k^*: w_1,..,w_m \geq 0, \sum\limits_{k=1}^mw_k \leq 1, \Vert x_k \Vert_1=1, 1\leq k \leq m  \Big\}$. Notice that $\Omega_1 \subset \Omega_2 \subset ..\subset\Omega_m \subset ..\subset\Omega $. By Minkowski-Carath\'eodory Theorem if $T \in \Omega$, then $T \in \Omega_{\mathrm{dim \;} S^n(\BC)+1}$. Therefore 
\begin{eqnarray*}
\Omega &=& \bigcup_{m \geq 1} \Omega_m=\Omega_1 \cup...\cup \Omega_{n^2+1}=\Omega_{n^2+1}\\
&=& \Big\{\sum_{k=1}^{n^2+1} t_k x_k x_k^*: \sum_{k=1}^{n^2+1} t_k=1, t_k \geq 0, \|x_k\|_1 = 1, \forall k, 1 \leq k \leq n^2+1 \Big\} 
\end{eqnarray*} 
We recall that the dimension of $S^n(\BC)$ as a real vector space over is $n^2$. As such, and since $X$  is compact, we conclude that  $\Omega$  as a convex hull of a compact set is compact.  

To show that int $\Omega \neq \emptyset$, take $\frac{1}{2n^2}I \in \Omega$. We prove that for  $0<r<\tfrac{1}{2n^2}$ we have the ball
\[B_r\left(\frac{1}{2n^2}I\right) =\Big\{ \frac{1}{2n^2}I+T: T=T^*; \Vert T \Vert_{op} < r \Big\} \subset \Omega.\] 
Let $T=\sum\limits_{k=1}^n \lambda_k v_k v_k^*$, $\Vert v_k \Vert_2=1$, and $|\lambda_k|\leq \Vert T\Vert_{op} < r.$ Now 
\begin{eqnarray*}
\frac{1}{2n^2}I+T &=& \frac{1}{2n^2} \sum\limits_{k=1}^n v_k v_k^*+\sum\limits_{k=1}^n \lambda_k v_k v_k^* \\
&=&  \sum\limits_{k=1}^n \left(\frac{1}{2n^2} +\lambda_k \right) \Vert v_k \Vert_1^2 \cdot \left(\frac{v_k}{\Vert v_k \Vert_1} \right)\cdot \left(\frac{v_k}{\Vert v_k \Vert_1} \right)^*.
\end{eqnarray*}
Also \[ \Vert v_k \Vert_1=\sum_{j=1}^n |v_{k,j}| \leq \left(\sum_{j=1}^n |v_{k,j}|^2 \right)^{\frac{1}{2}} \cdot \left(\sum_{j=1}^n 1 \right)^{\frac{1}{2}}= \sqrt{n} \Vert v_k \Vert_2=\sqrt{n}.  \]
Hence 
\begin{eqnarray*}
&& \|\frac{1}{2n^2}I+T \|_{1,1} \leq  \sum_{k=1}^n \left(\frac{1}{2n^2} +\lambda_k \right)\Vert v_k \Vert_1^2 \leq n \left(\frac{1}{2n^2} +r \right) n= \frac{1}{2}+rn^2 < 1
\end{eqnarray*} 
In addition, because $r<\tfrac{1}{2n^2}$ we conclude that $$\ip{(\tfrac{1}{2n^2}I+T)x}{x}\geq \|x\|^2(\tfrac{1}{2n^2}-r)\geq 0$$ for all $x\in \BC^n$. Consequently,  $\frac{1}{2n^2}I+T \geq 0$.
We conclude that $B_{r}\left(\frac{1}{2n^2}I\right) \subset \Omega$ where we use the norm $\|A\|_{1,1}$ for convenience. 
\end{proof}
By a similar argument, $\Gamma$ is also compact convex subset of $S^n_+(\BC).$

\subsection{Measures of optimality}\label{subsec:2.2}
We next introduce and study the properties of some quantities defined on $S^n$ and which will serve as measures of optimality of the rank one decompositions of matrices in $S^n_+$.  

\begin{defn}\label{deft:gauge}
For $A\in S^n_+$ let 
\begin{equation}\label{gmpls}
\gamma_{+}(A):=\inf_{A=\sum\limits_{n \geq 1}g_n g_n^*} \sum\limits_{n \geq 1} \Vert g_n \Vert^2_1.
\end{equation}
If $A\in S^n$ we let 
\begin{equation}\label{gm}
\gamma(A):=\inf_{A=\sum\limits_{n \geq 1} g_n h_n^*} \sum\limits_{n \geq 1} \Vert g_n \Vert_1 \Vert h_n \Vert_1,
\end{equation}
 and 
\begin{equation}\label{gm0}
\gamma_0(A):= \inf_{\substack{ A=B-C,\\
B, C \in S^n_+}} \left( \gamma_+(B)+\gamma_+(C)\right)=\inf_{A=\sum\limits_{n \geq 1} g_n g_n^* - \sum\limits_{k \geq 1} h_k h_k^*} \left( \sum\limits_{n \geq 1} \Vert g_n \Vert^2_1 +\sum\limits_{k \geq 1} \Vert h_k \Vert^2_1 \right).
\end{equation}

\end{defn}

We collect some of the properties of these functionals.

\begin{prop}\label{prop:subadd} The functionals given in Definition~\ref{deft:gauge} are sub-additive. In particular, the following statements hold.
\begin{enumerate}
\item[(a)] Given  $A, B\in S^n_+$ we have $\gamma_+(A+B) \leq \gamma_+(A)+\gamma_+(B)$
\item[(b)] Given  $A, B\in S^n$ we have $\gamma(A+B) \leq \gamma(A)+\gamma(B)$
\item[(c)]  Given  $A, B\in S^n$ we have $\gamma_0(A+B) \leq \gamma_0(A)+\gamma_0(B)$
\end{enumerate}

In addition, if $a \geq 0$, we have $\gamma_+(aA)=a \gamma_+(A)$ when $A\in S^n_+$, and 

\[\left\{ \begin{array} {r@{\quad = \quad}l} 
\gamma(aA)& |a|  \gamma(A)\\
\gamma_0(aA)& |a|  \gamma_0(A)
\end{array}\right. \]
for $A\in S^n$ and $a\in\mathbb{R}$. 
\end{prop}

\begin{proof}
Let $\epsilon>0$ and choose $\{g_k\}_{k\geq1}\subset \BC^n$ and $\{h_k\}_{k\geq1}\subset \BC^n$ such that 
\[ \left\{  \begin{array} {r@{\quad \leq \quad}l}
\sum_{k\geq 1}\|g_k\|_1^2& \gamma_+(A)+\epsilon/2\\
\sum_{k\geq 1}\|h_k\|_1^2& \gamma_+(B)+\epsilon/2
\end{array}\right. \] with $A=\sum_{k\geq 1}g_kg_k^*$ and $B=\sum_{k\geq 1}h_kh_k^*.$
It follows that $$A+B=\sum_{k\geq 1}g_kg_k^*+\sum_{k\geq 1}h_kh_k^*=\sum_{\ell\geq 1}f_\ell f_\ell^*,$$ after reindexing. Furthermore, $$\sum_{\ell\geq 1}\|f_\ell\|_1^2=\sum_{k\geq 1}\|g_k\|_1^2+\sum_{k\geq 1}\|h_k\|_1^2\leq \gamma_+(A)+\gamma_+(B)+\epsilon.$$ 
The rest of the statements are proved in a similar manner, so we omit the details. 
\end{proof}

The next result gives a comparison among the quantities defined above.  

\begin{prop}\label{prop:compare} For any $A \in S^n$ the following statements hold. 
\begin{enumerate}
\item[(a)]  $\gamma(A) \leq \gamma_0(A) \leq 2\gamma(A)$.
\item[(b)] $ \Vert A \Vert_{\mathcal{I}_1} \leq \Vert A \Vert_{1,1} \leq \gamma_0(A)\leq 2 \gamma(A).$ If in addition, we assume that  $A \in S^n_+$ then we have 
\[ \Vert A \Vert_{\mathcal{I}_1} \leq \Vert A \Vert_{1,1} \leq \gamma_0(A)\leq  \gamma_+(A). \]
\end{enumerate}
\end{prop}

\begin{proof}
\begin{enumerate}
\item[(a)] Let $A\in S^n$ such that $A=A^*=\sum\limits_{k \geq 1} g_k g_k^* - \sum\limits_{k \geq 1} h_k h_k^*$. Then,
$$\gamma(A)\leq \sum_{k\geq 1}\|g_k\|_1^2+\sum_{k\geq 1}\|h_k\|_1^2.$$ Consequently, $\gamma(A)\leq \gamma_0(A)$. 

Fix $\eps>0$ and let $\{g_k\}_{k=1}^M,\{h_k\}_{k=1}^M$ be such that $A=\sum_{k=1}^M g_k h_k^*$ and 
$$\sum_{k=1}^M \norm{g_k}_1\norm{h_k}_1\leq \gamma(A)+\eps.$$ 
Furthermore, rescale $g_k$ and $h_k$ so that $\norm{g_k}_1=\norm{h_k}_1$. 

Let $x_k=\frac{1}{2}(g_k+h_k)$ and $y_k=\frac{1}{2}(g_k-h_k)$. Then
\[ \sum_{k=1}^M x_kx_k^* - \sum_{k=1}^M y_ky_k^* = \frac{1}{2}\sum_{k=1}^M g_kh_k^* + \frac{1}{2}\sum_{k=1}^M h_kg_k^* = A\]
Note also $\norm{x_k}_1 \leq \norm{g_k}_1=\norm{h_k}_1$ and $\norm{y_k}_1\leq \norm{g_k}_1=\norm{h_k}_1$.
Thus
\[ \gamma_{0}(A)\leq \sum_{k=1}^M \norm{x_k}_1^2 + \sum_{k=1}^M\norm{y_k}_1^2 \leq 2\sum_{k=1}^M \norm{g_k}_1^2\leq 2\gamma(A)+2\eps.\]
Since $\eps>0$ was arbitrary, the second inequality follows.

\item[(b)] Since $\Vert A \Vert_{\mathcal{I}_1} = \max_{U\in U(n)}Real\, \mathrm{tr}(AU)$, let $U_0\in U(n)$ denote the unitary that achieves the maximum and makes the trace real. Then
    $$ \Vert A \Vert_{\mathcal{I}_1}
    =\mathrm{tr}(AU_0) = \sum_{k=1}^n\sum_{\ell=1}^n A_{k\ell}(U_0)_{\ell k}\leq \left(\sum_{k=1}^n\sum_{\ell=1}^n|A_{k\ell}|\right)\cdot\left(
\max_{k}\max_{\ell}|(U_0)_{\ell k}|\right) 
    \leq \sum_{k=1}^n\sum_{\ell=1}^n|A_{k\ell}|=\Vert A \Vert_{1,1} .$$    
Suppose that  $A\in S^n_+$ and let $\epsilon>0$. Choose $\{g_k\}_{k\geq 1}\subset \BC^n$ such that  $A=\sum_{k\geq 1}g_kg_k^*$ and 
$$\sum_{k\geq 1}\|g_k\|_1^2<\gamma_{+}(A)+\epsilon.$$ 
It follows that 
$$\gamma_0(A)\leq \sum_{k\geq 1}\|g_k\|_1^2<\gamma_{+}(A)+\epsilon.$$ 
\end{enumerate}

\end{proof}
The upper bound $2\gamma(A)$ is tight as we show in Proposition \ref{prop:prop1}. 
We next show that $\|\cdot \|_{1,1}$ and $\gamma(\cdot)$ are identical on $S^n.$ 

\begin{lem}\label{lem:gamequality} For any $A\in S^n$ we have 
$\Vert A \Vert_{1,1}=\gamma(A).$ Consequently, $(S^n, \gamma)$ is a normed vector space. 
\end{lem}

\begin{proof} Let $A\in S^n$ and $\epsilon>0$. Choose $\{g_j\}_{j\geq 1}, \{h_j\}_{j\geq 1}\subset \BC^n$ such that 
 $A=\sum\limits_{j} g_j h_j^*$ with $\sum\limits_{j} \Vert g_j \Vert_1 \cdot \Vert h_j \Vert_1 \leq \gamma(A)+\epsilon$.
 It follows that 
\[ \Vert A \Vert_{1,1}=\sum\limits_{i,j} |A_{i,j}|=\Vert \sum\limits_{j} g_j h_j^* \Vert_{1,1} \leq \sum\limits_{j} \Vert g_j h_j^* \Vert_{1,1} \leq \sum\limits_{j} \Vert g_j \Vert_1 \cdot \Vert h_j \Vert_1 \leq \gamma(A)+\epsilon. \]
Thus $\Vert A \Vert_{1,1} \leq \gamma(A).$ 

On the other hand, for $A\in S^n$ we can write: $A=(A_{i,j})_{i,j}=(\sum\limits_{j}(A_{i,j}))_i \cdot \delta_i^T, $ then
\[ \gamma(A) \leq \sum_j \Vert A_{i,j} \Vert_1 \cdot \Vert \delta_i \Vert_1=\sum_{i,j} | A_{i,j} |=\Vert A \Vert_{1,1}.  \]
Therefore $\Vert A \Vert_{1,1}=\gamma(A).$
\end{proof}

In fact,  $\gamma_0$ defines also a norm on $S^n$. More precisely, we have the following result.

\begin{prop}\label{prop:prop1}
 $(S^n, \gamma_0)$ is normed vector space.  Furthermore, $\gamma_0$ is Lipschitz with constant $2$ on $ S^n $:
\begin{equation}
\sup_{A,B\in S^n,A\neq B}\frac{|\gamma_0(A)-\gamma_0(B)|}{\|A-B\|_{1,1}} = 2.  
\label{Lip-gamma0}
\end{equation}
\end{prop}

\begin{proof} We have already established in Proposition~\ref{prop:subadd} that $\gamma_0$ satisfies the triangle inequality and is homogenous. Furthermore, suppose that $\gamma_0(A)=0$. It follows that $A=0$. 

For the last part, let  $A,B \in S^n $. We have 
$$\gamma_0(B) = \gamma_0(B-A+A) \le \gamma_0(B-A) + \gamma_0(A)$$
$$\gamma_0(A) = \gamma_0(B-B+A) \le \gamma_0(B)+\gamma_0(-B+A)$$

So $| \gamma_0(B) - \gamma_0(A) | 
\le \gamma_0(B-A) 
\le 2 \gamma(B-A)
\le 2 \|B-A\|_{1,1}.$

To show the Lipschitz constant is exactly 2 (and hence the upper bound 2 is tight in Proposition \ref{prop:compare}(a) ) consider the matrix
\[ A = \left[ \begin{array}{cc}
0 & 1 \\ 1 & 0 \end{array} \right]. \]
Note $\|A\|_{1,1}=2$. For any decomposition $A=B-C$ with $B,C\in S^2_+$ we have
\[ B =\left[ \begin{array}{cc} a & b \\ b & c \end{array} \right]~,~
C = \left[ \begin{array}{cc} a & e \\ e & c \end{array} \right] \]
with $a,c\geq 0$ and $b-e=1$. Then
\[ \gamma_0(A)\geq \gamma_+(B)+\gamma_+(C)\geq \gamma(B)+\gamma(C) = 2a+2|b|+2|1-b|+2c \geq 4|b|+4|1-b|\geq 4, \]
thanks to $ac \geq b^2$ and $ac \geq e^2.$
On the other hand
\[ A = 
\frac{1}{2}\left[ \begin{array}{c} 1 \\ 1 \end{array}\right] 
\left[ \begin{array}{cc} 1 & 1 \end{array} \right]
-
\frac{1}{2}\left[ \begin{array}{c} 1 \\ -1 \end{array}\right] 
\left[ \begin{array}{cc} 1 & -1 \end{array} \right]
\]
which certifies $\gamma_0(A)=4$. The proof is now complete.

\end{proof}

We have now established that $\gamma_0$, $\gamma=\|\cdot\|_{1,1}$ are equivalent norms on $S^n$. In adition, we proved in Proposition~\ref{prop:compare} that $\gamma(A)= \|A\|_{1,1}\leq \gamma_+(A)$ for  $A \in S^n_{+}$. A natural question that arises is whether a converse estimate holds. More precisely, the rest of the chapter will be devoted to investigating the following questions.

\begin{question}\label{quest:1}
Fix $n\geq 2$. 
\begin{enumerate}
\item Does there exist a constant $C>0$, independent of $n$ such that for all $A\in S^{n}_{+}$, we have 
$$\gamma_+(A) \leq C \cdot \Vert A \Vert_{1,1}.$$
\item For a given  $A\in S^n_+$, give an algorithm to find   $\{h_1, h_2,..,h_M\}$ such that $A=\sum_{k=1}^M h_k h^*_k$ with $$\gamma_+(A)=\sum_{k=1}^M \Vert h_k \Vert^2_1.$$
\end{enumerate}
\end{question}

We begin by justify why the second question makes sense. In particular, we prove that $\gamma_+(A)$ is achieved for a certain decomposition.

\begin{thm}\label{thm:main1} Given $T \in S^n_+$, 
\[ \gamma_{+}(T)=\inf_{T=\sum\limits_{k \geq 1}g_k g_k^*} \sum\limits_{k \geq 1} \Vert g_k \Vert^2_1=\min_{T=\sum\limits_{k =1}^{n^2+1}g_k g_k^*} \sum\limits_{k = 1}^{n^2+1} \Vert g_k \Vert^2_1\] for some $\{g_k\}_{k=1}^{n^2+1}\subset \BC^n.$
\end{thm}

\begin{proof}

Let $T \in S^n_+(\BC)$, 
\[ \gamma_{+}(T)=\inf_{T=\sum\limits_{k \geq 1}g_k g_k^*} \sum\limits_{k \geq 1} \Vert g_k \Vert^2_1. \]
Assume $T \neq 0$, then $\gamma_{+}(T)> 0$. Let $\tilde{T}=\frac{T}{\gamma_{+}(T)}$, 
\[ \tilde{T}=\frac{1}{\gamma_{+}(T)} \sum\limits_{k \geq 1} g_k g_k^* =\sum\limits_{k \geq 1} \frac{\Vert g_k \Vert^2_1}{\gamma_{+}(T)} \cdot \left( \frac{g_k}{\Vert g_k \Vert_1} \right) \cdot \left( \frac{g_k}{\Vert g_k \Vert_1} \right)^*=\sum\limits_{k \geq 1} w_k \cdot e_k e_k^*,\] 
where $w_k=\frac{\Vert g_k \Vert^2_1}{\gamma_{+}(T)}$ and $e_k=\frac{g_k}{\Vert g_k \Vert_1}$. Hence $\sum\limits_{k \geq 1} w_k= \frac{1}{\gamma_{+}(T)} \sum\limits_{k \geq 1} \Vert g_k \Vert_1^2=1$ and $\Vert e_k \Vert_1=1.$ Therefore $\gamma_{+}(\tilde{T})=1$.  It follows that $\tilde{T} \in \Gamma.$ 

By Minkowski-Carath\'eodory Theorem~\ref{thm:MCT} 
\[\tilde{T}=\sum_{k=1}^{n^2+1} w_k \cdot e_k e^*_k, w_k \geq 0, \sum_{k=1}^{n^2+1}w_k=1. \]
Therefore \[ \gamma_{+}(T)=\min_{\sum\limits_{k =1}^{n^2+1}g_k g_k^*} \sum\limits_{k = 1}^{n^2+1} \Vert g_k \Vert^2_1. \]
\end{proof}

The next question one could ask is how to find an optimal decomposition for $A\in S^n_+$ that achieves the value $\gamma_+(A)$. The following technical tool will be useful in addressing this question, at least for small size matrices.

\begin{thm}\label{thm:rankdecrea}
Suppose that $A\in S^n_+(\BC)$ and $ y\in \BC^n$. Then   $A-yy^{*} \in S^n_+(\BC)$  if
and only if there exists $x\in \BC^n$ such that $y=Ax$ and $\ip{Ax}{x}\leq1.$
When equality holds, then $A-yy^{*}$ will have rank one less than
that of $A.$
\end{thm}

\begin{proof} The case $y=0$ is trivial, so we can assume without loss of generality that $y\neq 0$.

Suppose there exists a vector $y$ such that $y=Ax$ and
$\ip{Ax}{x}\leq1.$ For any vector $z$ and observe that  $|\ip{Ax}{z}|^2\leq \ip{Ax}{x} \ip{Az}{z}.$ Consequently, 
\[ \ip{(A-yy^{*})z}{z}=\ip{Az}{z}-|\ip{Ax}{z}|^2\geq \ip{Az}{z}-\ip{Ax}{x}\ip{Az}{z}=\ip{Az}{z}(1-\ip{Ax}{x})\geq 0.\]
When $\ip{Ax}{x}=1,$ we $\ip{(A-yy^*)x}{x}=\ip{Ax}{x}-|\ip{y}{x}|^2=\ip{Ax}{x}-|\ip{Ax}{x}|^2=0$. 
It follows that  $x\in\mathcal{N}(A-yy^{*}).$ Combining the fact that
$x\notin\mathcal{N}(A),$ we have $\text{rank}(A-yy^{*})<\text{rank}(A).$

For the converse, suppose that $A-yy^{*}$ is positive semidefinite, where $y\in \BC^n$. 
Write $y=Ax+z$ where $x\in \BC^n$ and $Az=0$. It follows that 
\[
\ip{(A-yy^{*})z}{z}=-|\ip{y}{z}|^2\leq 0
\]
 with equality only if $0=\ip{z}{y}=\ip{z}{Ax+z}=\ip{z}{z} $ which implies
$z=0.$ In addition, 
\[
\ip{(A-yy^{*})x}{x}=\ip{Ax}{x}-\ip{Ax}{x}^{2}\geq0
\]
 implies $\ip{Ax}{x}\leq1.$
\end{proof}

The following result follows from Theorem~\ref{thm:rankdecrea}

\begin{cor}\label{cor:eqgampl} For any $A\in S^n_+(\BC)$ we have 
\begin{align*}
\gamma_{+}(A) & =\min_{\ip{Ax}{x}\leq1,x\neq0}\gamma_{+}(A-Axx^{*}A)+\|Ax\|_{1}^{2}\\
 & \leq\min_{\ip{Ax}{x}=1}\gamma_{+}(A-Axx^{*}A)+\|Ax\|_{1}^{2}.
\end{align*}

\end{cor}

\begin{proof}
Let $A\in S^n_+$ and $0\neq x\in \BC^n$ such that $\ip{Ax}{x}\leq 1$. Then by Theorem~\ref{thm:rankdecrea} and Proposition~\ref{prop:subadd}(a), we see that 
$$\gamma_{+}(A) \leq \min_{\ip{Ax}{x}\leq1,x\neq0}\gamma_{+}(A-Axx^{*}A)+\|Ax\|_{1}^{2}$$
On the other hand, let $A=\sum_{k=1}^Nu_ku_k^*$ be an optimal decomposition, that is $\gamma_+(A)=\sum_{k=1}^N\|u_k\|_1^2$. Since $A-Axx^{*}A\in S^n_+$, we can write $A-Axx^{*}A=\sum_{k=1}^nv_kv_k^*$. Hence, $A=\sum_{k=1}^nv_kv_k^* +Axx^*A$ and by the optimality, we see that $$\gamma_+(A-Axx^*A)+ \|Ax\|_1^2\leq \sum_{k=1}\|v_k\|_1^2+\|Ax\|_1^2\leq \gamma_+(A)$$

\end{proof}

We recall that $ \Omega=\text{co}  \left(X \cup \{0\} \right)$ where $X=\{xx^*\, : x\in \BC^n\, \|x\|_1=1\}$. We now give a characterization of $\Omega$ in terms of $\gamma_+$ that is equivalent to the one proved in 
 Lemma~\ref{lem:omega}. 

\begin{lem}\label{lem:omegabis} Using the notations of Lemma~\ref{lem:omega}, the following result holds. 
$\Omega=\{T \in S^n_+(\BC): \gamma_+(T) \leq 1\}$.
\end{lem}

\begin{proof}
Let $T \in \{T \in S^n_+(\BC): \gamma_+(T) \leq 1\}$. Then \[T =\sum_{k=1}^{n^2+1}g_k g_k^*=\sum_{k=1}^{n^2+1} w_k X_k X_k^*, \]
where $w_k=\Vert g_k \Vert^2_1$ and $X_k=\frac{g_k}{\Vert g_k \Vert_1}$. Therefore $\gamma_+(T)=\sum\limits_{k=1}^{n^2+1} w_k \leq 1.$ Hence \[T=\sum_{k=1}^{n^2+1} w_k X_k X_k^* +\left(1-\gamma_+(T)\right) \cdot 0 \in  \Omega.\]
Conversely, let $T \in \Omega$. Then $T=\sum\limits_k w_k X_k X_k^*$, $w_k \geq 0$, and $\sum\limits_k w_k \leq 1.$ Hence 
\[ \gamma_+(T) \leq \sum_k w_k \cdot \gamma_+(X_k X_k^*)=\sum_k w_k \leq 1. \]
\end{proof}

In fact, $\gamma_+$ can be identified with the following  gauge-like function  $\varphi_{\Omega}: S^n_+(\BC)\rightarrow \BR$ defined as follows: \[ \varphi_{\Omega}(T)=\inf\{t>0: T \in t \Omega\}. \]

Let $\tau_T=\{t>0: T \in t \Omega\}$. Then $\tau_T$ is nonempty, since $\frac{T}{\gamma_+(T)} \in \Gamma \subset \Omega \Rightarrow T \in \gamma_+(T) \Omega \Rightarrow \gamma_+(T) \in \tau_T.$ Therefore $\varphi_{\Omega}(T) \leq \gamma_+(T).$ In fact, the following stronger result holds.

\begin{lem}\label{lem:gauge}
For each $T\in S^n_+$ we have 
$\varphi_{\Omega}(T) = \gamma_+(T)$
\end{lem}

\begin{proof}
We need to prove $ \gamma_+(T) \leq \varphi_{\Omega}(T)$. If $t \in \tau_T$, then $\frac{T}{t} \in \Omega,$
\[ \frac{T}{t}=\sum_{k=1}^{n^2+1} w_k x_k x^*_k, w_1,..,w_{n^2+1} \geq 0, \sum_{k=1}^{n^2+1} w_k \leq 1, \Vert x_k \Vert_1=1, \forall k.  \]
\[ T=\sum_{k=1}^{n^2+1} t w_k x_k x^*_k=\sum_{k=1}^{n^2+1} g_k g_k^*, \] where $g_k=\sqrt{tw_k}x_k.$ Now $ \gamma_+(T) \leq \sum\limits_{k=1}^{n^2+1} t w_k=t \sum\limits_{k=1}^{n^2+1} w_k \leq t \Rightarrow \gamma_+(T) \leq \varphi_{\Omega}(T).$
\end{proof}

\begin{rem} If follows that $\varphi_{\Omega}$ is also positively homogeneous and sub-additive, hence convex. However, we point out that $\varphi_\Omega$ is not a Minkowski gauge function since  $\Omega$ does not include a neighborhood of $0$. 

\end{rem}

We close this section with a discussion of some regularity properties of $\gamma_+$.

\begin{thm}\label{thm:lipsch}
Fix $\delta > 0$. Let $C_\delta=\{T \in S^n_+: T \geq \delta I, tr(T)\leq 1\}$, then $\gamma_+: C_\delta \rightarrow \BR$ is Lipschitz continuous on $C_\delta$ with Lipschitz constant $(n/{\delta})+n^{3/2}$.
\end{thm}

\begin{proof}
We show that, $\forall \ T_1, T_2 \in C_\delta$,
\[ \left|\gamma_+(T_1)-\gamma_+(T_2)\right| \leq \left(\frac{n}{\delta}+n^2 \right) \Vert T_1-T_2 \Vert. \]
Define \[ \tilde{T}=T_2+\frac{\delta}{\Vert T_2 - T_1 \Vert} (T_2-T_1).\]
Then \[\lambda_{\min}(\tilde{T}) \geq \lambda_{\min}(T_2)- \left\Vert \frac{\delta}{\Vert T_2 - T_1 \Vert} (T_2-T_1) \right\Vert =\lambda_{\min}(T_2) -\delta \geq 0.\] Consequently, $ \tilde{T} \in S^n_+.$

Now \[ T_2=\frac{\delta}{ \delta + \Vert T_2 - T_1 \Vert} T_1+\frac{\Vert T_2 - T_1 \Vert}{ \delta + \Vert T_2 - T_1 \Vert} \tilde{T}  .\]
The convexity of $\gamma_+$ yields
\[ \gamma_+(T_2) \leq \frac{\delta}{ \delta + \Vert T_2 - T_1 \Vert} \gamma_+(T_1)+ \frac{\Vert T_2 - T_1 \Vert}{ \delta + \Vert T_2 - T_1 \Vert} \gamma_+(\tilde{T}), \]
which implies
\begin{equation}\label{eq16}
\gamma_+(T_2)- \gamma_+(T_1)\leq \frac{\Vert T_2 - T_1 \Vert \left( \gamma_+(\tilde{T}) - \gamma_+(T_1) \right)}{ \delta + \Vert T_2 - T_1 \Vert} .
\end{equation}
We have
\begin{equation}\label{eq13}
\gamma_+(\tilde{T}) \leq n \cdot tr(\tilde{T})=n \cdot \left[ tr (T_2)+\delta \cdot tr \left( \frac{T_2-T_1}{\Vert T_2 - T_1 \Vert} \right)  \right]\leq n \cdot tr (T_2) +\delta n^{3/2} .
\end{equation}
 \begin{equation}\label{eq14}
 \gamma_+(T_1) \geq \Vert T_1 \Vert_{1,1}=\sum_{i,j} |(T_1)_{i,j}| \geq  tr(T_1) \geq n \delta.
 \end{equation}
Using equations (\ref{eq13}) and (\ref{eq14}), we get
\begin{equation}\label{eq15}
\gamma_+(\tilde{T}) -\gamma_+(T_1) \leq n \cdot tr (T_2) +\delta n^{3/2} - n \delta \leq n \cdot tr (T_2) +\delta n^{3/2}.
\end{equation}
Now
\begin{eqnarray}\label{eq17}
 \gamma_+(T_2)- \gamma_+(T_1) & \leq & \frac{\Vert T_2 - T_1 \Vert }{\delta} \left( \gamma_+(\tilde{T}) - \gamma_+(T_1) \right) \leq  \Vert T_2 - T_1 \Vert \left[ \frac{n}{\delta} \cdot tr (T_2) + n^{3/2} \right]  \nonumber \\
& \Rightarrow & \frac{\gamma_+(T_2)- \gamma_+(T_1)}{\Vert T_2 - T_1 \Vert} \leq \frac{n}{\delta} \cdot tr (T_2) + n^{3/2}.
\end{eqnarray}
Similarly 
\begin{equation}\label{eq18}
\frac{\gamma_+(T_1)- \gamma_+(T_2)}{\Vert T_1 - T_2 \Vert} \leq \frac{n}{\delta} \cdot tr (T_1) + n^{3/2}.
\end{equation}
Therefore 
\begin{equation}\label{eq19}
\frac{|\gamma_+(T_1)- \gamma_+(T_2)|}{\Vert T_1 - T_2 \Vert} \leq \frac{n}{\delta} \cdot \max{\left( tr (T_1), tr (T_2) \right)} + n^{3/2} \leq \frac{n}{\delta}+n^{3/2}.
\end{equation}
\end{proof}

In fact, we can prove a stronger result if we restrict to $S^n_{++}$.

\begin{cor}\label{thm:main2}
$\gamma_+: S^n_{++}(\BC) \rightarrow \BR$ is continuous. Further, let $T \in S^n_{++}(\BC)$ and $\delta = \frac{1}{2} \lambda_{\min}(T)>0$. Then for every $S \in S^n_{++}(\BC)$ with $\Vert T-S \Vert \leq \delta$,
\[ \frac{|\gamma_+(T)- \gamma_+(S)|}{\Vert T - S \Vert} \leq \frac{n}{\delta} \cdot tr(T)+ 2n^{3/2} .\]
\end{cor}

\begin{proof}
Let $T \in S^n_{++}(\BC)$ and $\delta = \frac{1}{2} \lambda_{\min}(T)>0$.  For any $S \in S^n_{++}(\BC)$ with $\Vert T-S \Vert \leq \delta$, and every $x\in \BC^n$ we have that 
$$\ip{Sx}{x}=\ip{(S-T)x}{x}+\ip{Tx}{x}\geq (-\delta+\lambda_{min}(T))\|x\|^2=\delta \|x\|^2.$$ 
Using this~\eqref{eq19} becomes

$$\frac{|\gamma_+(T)- \gamma_+(S)|}{\Vert T - S \Vert} \leq \frac{n}{\delta} \cdot \max{\left( tr (T), tr (S) \right)} + n^{3/2}.$$ However, $tr(S)\leq tr(T)+\sqrt{n}\delta$. Therefore,
$$\frac{|\gamma_+(T)- \gamma_+(S)|}{\Vert T - S \Vert} \leq \frac{n}{\delta} \cdot tr(T)+ 2n^{3/2}.$$

\end{proof}

\section{Finding optimal rank one decomposition for some special classes of matrices}\label{sec:3}
In this section we consider several classes of matrices in $S^n_+$ for which the answer to Question~\ref{quest:1} is affirmative.

\subsection{Diagonally dominant matrices}\label{subsec:3.1}
Recall that a matrix $A\in S^n_+(\BC)$ is said to be diagonally dominant if $A_{ii}\geq \sum_{j=1}^n|A_{ij}|$ for each $i=1,2, \hdots, n$. If the inequality is strict for each $i$, we say that the matrix is strictly diagonally dominant.  The following result can be proved for any diagonally dominant matrix in $S^n_+$.

\begin{thm}\label{thm:main3}
Let $A \in S_+^n $ be a diagonally dominant matrix. Then $\gamma(A)=\gamma_0(A)=\gamma_+(A)$.
\end{thm}

\begin{proof} Let $e_i = (0,...,0,1,0,...,0)$ and $u_{ij}(x) = (0,...,\sqrt{x},...,\overline{\sqrt{x}},...,0)$. Given a diagonally dominant matrix $A$, we consider the following decomposition of $A$ (\cite{BarCar75})
$$ A = \sum_{i<j} u_{ij}(A_{ij}) u_{ij}(A_{ij})^* + \sum_i ( A_{ii} - \sum_{j \in \{1,...,n\} \backslash \{i\} } |A_{ij}| ) e_i e_i^* .$$
It follows that 
\begin{align*}
\gamma_+(A) & \le \sum_{i<j} 4|A_{ij}| + \sum_i ( A_{ii} - \sum_{j \in \{1,...,n\} \backslash \{i\} } |A_{ij}| )\\
&= \sum_{i<j} 4|A_{ij}| + \sum_i  A_{ii} - \sum_i\sum_{j \in \{1,...,n\} \backslash \{i\} } |A_{ij}| \\
& = \sum_{i<j} 4|A_{ij}| + \sum_i  A_{ii} - \sum_{i<j} 2|A_{ij}| \\
& = \| A \|_{1,1}.
\end{align*}
\end{proof}

The case of diagonally dominant matrices is a particular case of the following more general decomposition result:

\begin{thm}\label{lem:special}
Assume $A\in S^n_+$ admits a decomposition
\begin{equation}
A = \sum_{1\leq i<j\leq n} u_{ij}u_{ij}^* + \sum_{i=1}^n v_iv_i^* 
\end{equation}
where each $u_{i,j}$ has non-zero entries at most on positions $i$ and $j$, and each $v_i$ has non-zero entries at most on position $i$. 
 Then $\gamma_+(A)=\|A\|_{1,1}$.
\end{thm}

\begin{proof}

The hypothesis implies
\[ u_{ij} = \left[ \begin{array}{ccccccccccc}
0 & \cdots & 0 & \mbox{$c_{ij;i}$} & 0 & \cdots & 0 & \mbox{$c_{ij;j}$} & 0 & \cdots & 0 
\end{array} \right]^T \]
and
\[
v_i =  \left[ \begin{array}{ccccccc}
0 & \cdots & 0 & \mbox{$d_i$} & 0 & \cdots & 0
\end{array} \right]^T
\]
where $c_{ij;i}$ is on position $i$, $c_{ij;j}$ is on position $j$ and $d_i$ is on position $i$. Without loss of generality we can assume $d_i\in\BR$ 
and $c_{ij;i},c_{ij;j}\in\BC$.
We write $A=(a_{ij})_{i,j=1}^n$ where for $1\leq i<j\leq n$, $a_{ij} = c_{ij;i}\overline{c_{ij;j}}$, 
whereas for $1\leq i\leq n$,
\[ a_{ii}=d_i^2+\sum_{j=1}^{i-1}|c_{ji;i}|^2 + \sum_{j=i+1}^n |c_{ij;i}|^2. \]
These imply
\[ \sum_{1\leq i<j\leq n}\|u_{ij}\|_1^2 + \sum_{i=1}^n \|v_i\|_1^2 = \sum_{1\leq i<j\leq n}\left( |u_{ij;i}|+|u_{ij;j}| \right)^2 + \sum_{i=1}^n d_i^2
 = \sum_{1\leq i,j\leq n}|a_{i,j}| =\|A\|_{1,1}. \]
Now the proof is complete.
\end{proof}

\subsection{The cases for matrices in $S^n_+(\BC)$ for $n\in \{2,3\}$}\label{subsec:3.2}

\begin{prop}\label{prop:n2}
Suppose that $A\in S_{+}^2$, then $$\gamma_+(A)=\|A\|_{1,1}.$$

\end{prop}

\begin{proof}
If $A=uu^{*}$ is a rank $1$ matrix in $S_{+}^2$, the proof is straightforward.
Suppose $A\in S_{+}^2$ is rank $2$. $A=\begin{bmatrix}a&c\\ \bar{c}& b\end{bmatrix}$ with $ab-|c|^2>0$. Using the Lagrangian decomposition \cite{Ycart85} we can write 

$$A= \begin{bmatrix}\sqrt{a}\\ \tfrac{\bar{c}}{\sqrt{a}}\end{bmatrix}  \begin{bmatrix}\sqrt{a}& \tfrac{c}{\sqrt{a}}\end{bmatrix}
+  \begin{bmatrix}0\\ \sqrt{b- \tfrac{|c|^2}{a}} \end{bmatrix}\begin{bmatrix}0& \sqrt{b- \tfrac{|c|^2}{a}} \end{bmatrix}$$ The result then follows.
\end{proof}
For certain $3\times 3$ matrices the Lagrangian decomposition \cite{Ycart85} is optimal. In particular, we have the following result. 

\begin{prop}\label{prop:n3} Let $A\in S_{+}^3$ be of rank $2$ or $3$. If $$A=\begin{bmatrix}a&b&c\\ \overline{b}&d&e\\ \overline{c}&\overline{e}&f\end{bmatrix}$$ then
$$\gamma_+(A)\leq \|A\|_{1,1}+ \tfrac{2(|ae-\overline{b}c|+|b||c|-a|e|)}{a}.$$ In particular, if $|ae-\overline{b}c|+|b||c|=a|e|$ then 
$\gamma_+(A)=\|A\|_{1,1}$ and the Lagrangian decomposition (which in this case is the LDL factorization) is optimal.
\end{prop}

\begin{proof}
We first assume that $A$ has rank $3$. In this case, $A$ must be positive definite and $adf\neq 0$. Indeed, if one of the diagonal term, say $f=0$, then using the fact that $A\in S_+^3$ would implies that $df-|e|^2=-|e|^2>0$ which is impossible.

Let $$u_1=\tfrac{1}{\sqrt{a}}A\delta_1=\begin{bmatrix}\sqrt{a}\\ \tfrac{\overline{b}}{\sqrt{a}}\\  \tfrac{\overline{c}}{\sqrt{a}}\end{bmatrix},$$ where $\{\delta_i\}_{i=1}^3$ is the standard ONB  for $\BC^3$. By Theorem~\ref{thm:rankdecrea}, the matrix $A-u_1u_1^*$. In fact, in this case, this is a rank $2$ matrix given by

$$A-u_1u_1^*=\begin{bmatrix}0&0&0\\0& d-\tfrac{|b|^2}{a}& e-\tfrac{\overline{b}c}{a}\\ 0 &\overline{e}-\tfrac{\overline{c}b}{a} & f-\tfrac{|c|^2}{a}\end{bmatrix}$$ 

  Let $$u_2=\tfrac{1}{\sqrt{d-\tfrac{|b|^2}{a}}}(A-u_1u_1^*)\delta_2=\begin{bmatrix}0\\ \sqrt{d-\tfrac{|b|^2}{a}}\\\tfrac{\overline{e}-\tfrac{\overline{c}b}{a}}{\sqrt{d-\tfrac{|b|^2}{a}}}\end{bmatrix}.$$ It follows that
  $A-u_1u_1^*-u_2 u_2^*=u_3u_3^*$ where 
  $$u_3=\begin{bmatrix}0\\ 0\\ \sqrt{\tfrac{\text{det} A}{ad-|b|^2}}\end{bmatrix}.$$ Consequently, the Lagrange decomposition of $A$ is $A=u_1u_1^*+u_2 u_2^*+u_3u_3^*$ which implies that
  $$\gamma_+(A)\leq \sum_{k=1}^3\|u_k\|_1^2=\|A\|_{1,1}+ \tfrac{2(|ae-\overline{b}c|+|b||c|-a|e|)}{a}.$$

Now suppose that the rank of $A$ is $2$. In this case, it is possible for $adf=0$. However, only one of the diagonal element can be $0$. So assume that $f=0$, then we also get that $e=c=0$. In this case $$A\begin{bmatrix}a&b&0\\ \overline{b}&d&0\\ 0 &0 &0\end{bmatrix}$$ which reduces to Proposition~\ref{prop:n2}. Thus, we may assume without loss of generality that $adf\neq 0$. In this case, we can proceed as above. However, because the rank of the matrix $A$ is now $2$ we see that $A=u_1u_1^*+u_2 u_2^*$ and 
$$\gamma_+(A)\leq \|u_1\|_1^2+\|u_2\|_1^2 =\|A\|_{1,1}+ \tfrac{2(|ae-\overline{b}c|+|b||c|-a|e|)}{a}.$$
\end{proof}

\begin{rem}
\begin{enumerate}
\item If one of the off diagonal elements $b,$ or $c$ is $0$, then Proposition~\ref{prop:n3} shows that the Lagrange decomposition is optimal for $\gamma_+(A)$.
\item Suppose $n=4$ and let $V=\tfrac{1}{\sqrt{14}}\begin{bmatrix}1&0\\0&1\\1&-1\\1&1\end{bmatrix}$, and consider $$A=VV^{T} =\tfrac{1}{14} \begin{bmatrix} 1&0&1&1\\ 0&1&-1&1\\ 1&-1&2&0\\ 1&1&0&2\end{bmatrix}$$  
Then $A$ has rank $2$, and the  $\|A\|_{1,1}=1$. However, $\gamma_+(A)\neq \gamma(A)$.
\end{enumerate}

\end{rem}

%
%
%
%

\section{Numerics}\label{sec:4}

Here we inspect upper bounds of $\gamma_+(A)/\|A\|_{1,1}$ for A an N x N matrix with simulated data. We randomly generate symmetric positive definite matrices and compute upper bounds on $\gamma_+(A)/\|A\|_{1,1}$ with different decompositions of A. The first step is generating Gaussian distributed realizations in a matrix size N by N. Then by multiplying by its transpose, the result is symmetric positive semi-definite, denoted $A$. 
Let ${\mathcal A_N}$ denote a collection of 30 independent realizations of this random matrix.

We consider two factorizations of the matrix $A$: the LDL and the Eigen matrix decomposition. Specifically:
\[ LDL: ~~A= \sum_{k=1}^N v_k v_k^* \]
with $v_k$ vectors that have the top $k-1$ entries 0, and
\[ Eigen: ~~ A= \sum_{k=1}^n g_kg_k^* \]
where $\{g_1,...,g_n\}$ are the eigenvectors, each scaled by the corresponding eigenvalue's square-root. 
For each decomposition denote:
\[ J_{LDL}(A) = \sum_{k=1}^N \|v_k\|_1^2 ~~{\rm and}~~J_{Eigen}(A) = \sum_{k=1}^N \|g_k\|_1^2 \]
Let $F_{LDL}$ and $F_{Eigen}$ denote the worst upper bounds over the $N$ realization ensemble:
\[ F_{LDL}(N) = \max_{A\in {\mathcal A_N}} \frac{J_{LDL}(A)}{\|A\|_{1,1}} \]
\[ F_{Eigen}(N) = \max_{A\in {\mathcal A_N}} \frac{J_{Eigen}(A)}{\|A\|_{1,1}} \]
We plot these worst upper bounds after 30 realizations for various N in figure \ref{fig:exp1}.

In the same figure we plot the analytic approximations of these two curves using a square-root functions and a logarithmic function. The square-root function was scaled as $c\sqrt{N}$ to closely fit the Eigen decomposition bound, $F_{Eigen}(N)$. Numerically we obtained $c=4/5$.

From these plots we notice a clearly strictly increasing trend.  Furthermore, the LDL factorization produces a smaller (tighter) upper bound than the Eigen decomposition. On the other hand, as we show in Theorem \ref{thm:main1}, any optimal decomposition may take $N^2+1$ vectors. By limiting the number of vector to $N$ one should not expect to achieve the optimal bound $\gamma_{+}(A)$ with any decomposition.

\begin{figure}
  \includegraphics[width=\linewidth,trim={2cm 7cm 2cm 7cm},clip]{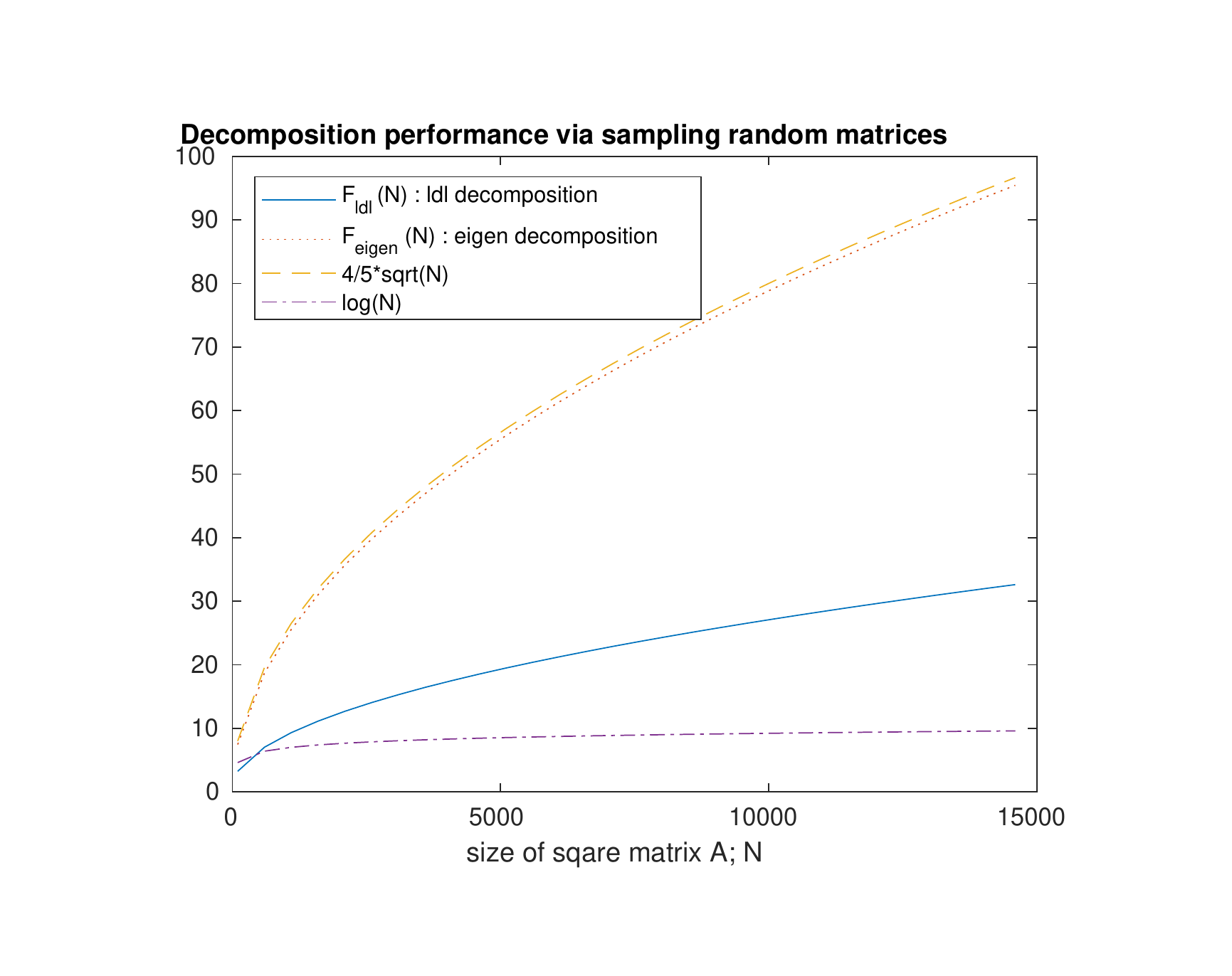}
  \caption{}
  \label{fig:exp1}
\end{figure}

\section*{Acknowledgments} R.~Balan was  partially supported by the National Science Fundation grant DMS-1816608 and Laboratory for Telecommunication Sciences under grant H9823031D00560049. 
K.~A.~Okoudjou was partially supported by   the U. S.\ Army Research Office  grant  W911NF1610008, 
 the National Science Foundation grant DMS 1814253, and an MLK  visiting professorship.





\bibliographystyle{plain}

\end{document}